\documentclass[12pt,a4paper,reqno]{amsart}
\usepackage{latexsym}
\usepackage{amssymb}
\usepackage{enumitem}

\usepackage{txfonts}

\def \zpizq{\left(}
\def \zcizq{\left[}
\def \zpder{\right)}
\def \zcder{\right]}

\def \za{\alpha}
\def \zb{\beta}
\def \zg{\gamma}
\def \zd{\delta}
\def \ze{\varepsilon}

\def \zh{\theta}

\def \zl{\lambda}
\def \zm{\mu}

\def \zp{\pi}

\def \zr{\rho}

\def \zf{\varphi}

\def \zq{\psi}
\def \zw{\omega}

\def \zF{\Phi}
\def \zQ{\Psi}

\def \zlma{\ell}

\def \zsu{\sum}

\def \zin{\cap}
\def \zing{\bigcap}
\def \zun{\cup}
\def \zung{\bigcup}

\def \zpu{\cdot}
\def \zpor{\times}
\def \zci{\circ}

\def \zmai{\geq}
\def \zco{\subset}

\def \zpe{\in}

\def \zeq{\equiv}

\def \znoi{\neq}

\def \znope{\not\in}

\def \zpar{\partial}
\def \zinf{\infty}
\def \zva{\emptyset}

\def \zfl{\rightarrow}

\def \zdbv{\parallel}
\def \z/{\over}

\newcommand {\CC}{\mathbb C}

\newcommand {\RR}{\mathbb R}
\newcommand {\ZZ}{\mathbb Z}
\newcommand {\NN}{\mathbb N}
\newcommand {\A}{\mathcal A}

\newcommand {\E}{\mathcal E}

\newcommand {\G}{\mathcal G}
\newcommand {\T}{\mathcal T}
\newcommand {\Zb}{\mathsf Z}
\newcommand {\ib}{\mathsf i}
\newcommand {\menos}{\setminus}

\newcommand {\co}{\colon}

\def\mylabel#1{\label{#1}}   
   
\newtheorem{theorem}{Theorem}[section]

\newtheorem*{theorem*}{Theorem} 
\newtheorem{lemma}[theorem]{Lemma}
\newtheorem{corollary}[theorem]{Corollary} 
\newtheorem*{corollary*}{Corollary}
\newtheorem{proposition}[theorem]{Proposition}  
  
\newtheorem{remark}[theorem]{Remark} 
\newtheorem{example}[theorem]{Example}  
\newtheorem*{example*}{Example}

\title{Primary singularities of vector fields on surfaces}

\author{M.W. Hirsch}
\address[M.W.~Hirsch]
{Department of Mathematics University of Wisconsin at Madison, 
University of California at Berkeley} 
\email[M.W.~Hirsch]{mwhirsch@chorus.net}

\author{F.J.~Turiel}
\address[F.J.~Turiel]
{Geometr{\'\i}a y Topolog{\'\i}a,
Facultad de Ciencias,
Campus de Teatinos, s/n,
29071-M{\'a}laga, Spain}
\email[F.J.~Turiel]{turiel@uma.es}

\thanks{}

\begin{document}

\begin{abstract}
Unless another thing is stated one works in the $C^\infty$ category and manifolds have
empty boundary. Let $X$ and $Y$ be vector fields on a manifold $M$. We say 
that $Y$ tracks $X$ if $[Y,X]=fX$ for some continuous function
$f\colon M\rightarrow\mathbb R$. A subset $K$ of the zero set ${\mathsf Z}(X)$  
is an essential block for $X$ if it is non-empty, compact, open 
in ${\mathsf Z}(X)$ and its Poincar\'e-Hopf index does not vanishes.
One says that $X$ is non-flat at $p$ if its $\infty$-jet at $p$ is non-trivial. 
A point $p$ of ${\mathsf Z}(X)$ is called a primary singularity of $X$ if any
vector field  defined about $p$ and tracking $X$ vanishes at $p$.
This is our main result: Consider an essential block $K$ of a vector 
field $X$ defined on a surface $M$. Assume that $X$ is non-flat at every point of $K$.
Then $K$ contains a primary singularity of $X$.
As a consequence, if $M$ is a compact surface with non-zero characteristic and 
$X$ is nowhere flat, then there exists a primary singularity of $X$.
\end{abstract}

\maketitle

\tableofcontents

\section{Introduction} \mylabel{secIn}
Whether a family of vector fields has a common singularity is a classical issue in 
dynamical systems. For instance, on a compact surface with non-vanishing
Euler characteristic there always exists a common zero provided that the vector
fields commute (Lima \cite{Lima64}) or if they span a finite-dimensional nilpotent
Lie algebra (Plante \cite{Plante86}).
On the existence of a common singularity for a family of commuting vector
fields in dimension $\zmai 3$ several interesting results are due to
Bonatti \cite{Bonatti92} (analitic in dimension $3$ and $4$) and
Bonatti \& De Santiago \cite{Bon-Sant2015} (dimension $3$).
For a complementary discussion on the existence of a common zero the reader 
is referred to the introduction of \cite{HiTu}.

In this paper one shows that on surfaces every essential block of a nowhere
flat vector field $X$ includes a point at which all vector fields tracking $X$ vanish
(see Theorem \ref{the-1} below). 

Throughout this work manifolds (without boundary) 
and their associated objects are real $C^\zinf$ unless 
another thing is stated. Consider a tensor $\T$ on a manifold $P$. Given $p\zpe P$
the {\em principal part} of $\T$ at $p$ means $j_p^n \T$ if $j_p^{n-1} \T=0$ but 
$j_p^n \T\znoi 0$, or zero if $j_p^\zinf \T=0$. The {\em order} of $\T$ at $p$  
 is $n$ in  the first case  and $\zinf$ in the second one.
One will say that $\T$ is {\em flat at} $p$ if its order at this point equals $\zinf$, and
{\em non-flat} otherwise. 
 
In coordinates about $p$ the principal part is identified to the first significant term of the 
Taylor expansion of $\T$ at $p$. Given a function $f$ such that $f(p)\znoi 0$, the 
principal part  of $f\T$ at $p$ equals that of $\T$ multiplied by $f(p)$.

$\Zb(\T)$ denotes the set of zeros of $\T$ and $\Zb_n (\T)$, where $n\zpe\NN'$
and $\NN':=\NN\zun\{\zinf\}$, the set of zeros of order $n$. (Here $\NN$ is  the
set of positive integers.) Notice that
$\Zb(\T)=\zung_{k\zpe\NN'} \Zb_n (\T)$ where the union is disjoint.

Consider a vector field $Y$ on $P$. $Y$ {\em tracks} $\T$ provided
$L_Y \T=f\T$ for some continuous function $f\co P\rightarrow\RR$, referred to as 
the {\em tracking function}. (When $\T$ is also a vector field this means
$[Y,\T]=f\T$.) A set  $\A$ of vector fields on $P$ tracks $\T$ provided each 
element of  $\A$ tracks $X$.

A point $p\zpe\Zb(\T)$ is a {\em primary singularity} of $\T$ if every vector field 
defined about $p$ that tracks $\T$ vanishes at $p$. Obviously isolated singularities
are primary. The notion of primary singularity is the fundamental new
concept of this work.

Let $X$ be a vector field on $P$. Consider an open set $U$ of $P$ with compact 
closure $\overline U$ such that $\Zb(X)\zin(\overline U\menos U)=\zva$. 
The {\em index} of $X$ on $U$, denoted by ${\ib}(X,U) \zpe\ZZ$, is defined as  
the Poincar\'e-Hopf index of any sufficiently close approximation $X'$ to $X|\overline U$ 
(in the compact open topology) such that 
$\Zb(X')$ is finite. Equivalently: ${\ib}(X, U)$ is the intersection number of $X|U$
with the zero section of the tangent bundle ({\sc Bonatti} \cite{Bonatti92}).  
This number is independent of the approximation, and is stable under perturbation 
of $X$ and replacement of $U$ by smaller open sets containing ${\Zb}(X)\zin U$.

A compact set $K\zco \Zb(X)$ is a {\em block} of zeros for $X$ (or an
{\em $X$-block}) provided $K$ is non-empty and relatively open in $\Zb(X)$,
that is to say provided  $K$ is non-empty and $\Zb(X)\menos K$ is closed in $P$.
Observe that a non-empty compact $K\zco\Zb(X)$  
is a $X$-block if and only if it has a precompact open neighborhood
$U\subset P$, called {\em isolating} for $(X,K)$, such that $\Zb(X)\zin\overline U=K$
(manifolds are normal spaces). This implies $\ib(X,U)$ is determined by $X$ and $K$, 
and does not depend on the choice of $U$.
The {\em index of $X$ at $K$} is $\ib_K (X):=\ib(X,U)$. The
$X$-block $K$ is {\em essential} provided $\ib_K (X)\ne 0$, which
implies $K\ne\varnothing$, and {\em inessential} otherwise.

If $P$ is compact, it is isolating for every vector field on $P$ and its set of zeros. 
Therefore, in this case, ${\ib}_{\Zb(X)}(X)=\ib(X, P) =\chi (P)$.

This is our main result, which will be proved in the Section \ref{secPr}.

\begin{theorem}\mylabel{the-1}
Consider an essential block $K$ of a vector field $X$ defined on a surface $M$.
Assume that $X$ is non-flat at every point of $K$. Then $K$ contains a primary
singularity of $X$.
\end{theorem}

As a straightforward consequence:

\begin{corollary}\mylabel{cor-1}
On a compact connected surface $M$ with $\chi(M)\znoi0$
consider a vector field $X$. Assume that $X$ is nowhere flat. Then there exists
a primary singularity of $X$. 
\end{corollary}

Moreover, four examples illustrating these results are given in Section \ref{secEx}.

\begin{remark}\mylabel{rem-2} {\rm $\,$

{\bf (a)} The hypothesis on the non-flatness of Theorem \ref{the-1} and 
Corollary \ref{cor-1} cannot be omitted as the following example shows.
On $S^2 \zco \RR^3$ consider the vector field 
$X=\zf(x_3 )(-x_2 \zpar/\zpar x_1 +x_1 \zpar/\zpar x_2 )$ where $\zf(0)=1$
and $\zf(\RR\setminus(-1/2,1/2))=0$. Then the vector fields
$Y=-x_2 \zpar/\zpar x_1 +x_1 \zpar/\zpar x_2$ and
$V=\zq(x_3 )(-x_3 \zpar/\zpar x_1 +x_1 \zpar/\zpar x_3 )$ where
$\zq(1)=\zq(-1)=1$ and $\zq([-3/4,3/4])=0$ track $X$ and
$\Zb(Y)\zin\Zb(V)=\zva$. Therefore $X$ has no primary singularity. 

{\bf (b)} Two particular cases of Theorem \ref{the-1} were already known,
namely: if $X$ and $K$ are as in the foregoing theorem and $\G$ is a finite-dimensional 
Lie algebra of vector fields on $M$  
that tracks $X$, then the the elements of $\G$ have a common
singularity in $K$ provided that $\G$ is supersolvable (Theorem 1.4 of \cite{HiGD})
or $\G$ and $X$ are analytic (real case of Theorem 1.1 of \cite{HiTu}). Thus these 
two results are generalized here.}
\end{remark}

For general questions on Differential Geometry readers are referred to \cite{KN}, and for
those on Differential Topology to \cite{HR}.

\section{Other results} \mylabel{secRe}
One will need:

\begin{lemma}\mylabel{lem-1}
On a manifold $P$ of dimension $m\zmai 1$ consider a vector field $X$ of finite order 
$n\zmai 1$  at a point $p$. Then for almost every $v\zpe T_p P$ there exists a vector 
field $U$ defined around  $p$ such that $U(p)=v$ and the $n$-times iterated bracket 
$[U,[U,\dots[U,X]\dots]]$ does not vanish at $p$.  
\end{lemma}

\begin{proof}
It suffices to prove the result for $0\zpe\RR^m$ and a non-vanishing $n$-homogeneous 
polynomial vector field $X=\zsu_{\zlma=1}^m Q_\zlma \zpar/\zpar x_\zlma$. Up to a 
change of the order of  the coordinates, we may suppose $Q_1 \znoi 0$.

Given $a=(a_1 ,\dots,a_m )\zpe\RR^m$ set 
$U_a \co=\zsu_{\zlma=1}^m a_\zlma \zpar/\zpar x_\zlma$.
It suffices to show that for almost any $a\zpe\RR^m -\{0\}$ one has
$(U_a \zpu\zpu\zpu U_a \zpu Q_1 )(0)\znoi 0$, which is equivalent to show that 
the restriction of $Q_1$ to the vector line spanned by $a$ does not vanish identically.
But this last assertion is obvious. 
\end{proof}

Given a vector field $V$ on a manifold $P$, a set $S\subset P$ is {\em $V$-invariant} 
if it contains the orbits under $V$ of its points.

\begin{proposition}\mylabel{pro-1}
Consider two vector fields $X,Y$ on a surface $M$. Assume that $Y$ tracks $X$ 
with tracking function $f$. Then each set $\Zb_n (X)$, $n\zpe\NN'$, is
$Y$-invariant. 

Moreover $f$ is differentiable on the open set
$$[M\menos\Zb(X)]\zun[(\Zb(X)\menos\Zb_\zinf (X))\zin(M\menos\Zb(Y))].$$ 
\end{proposition}

This result is a consequence of the following two lemmas.

\begin{lemma}\mylabel{lem-2}
Under the hypotheses of Proposition \ref{pro-1} consider $p\zpe\Zb_n (X)$, 
$n<\zinf$, such that $Y(p)\znoi 0$. One has:
\begin{enumerate}[label={\rm (\alph{*})}]
\item\mylabel{lem-2a}
$f$ is differentiable around $p$.
\item\mylabel{lem-2b}
Let $\zg\co(a,b)\zfl M$ be an integral curve of $Y$ with $\zg(t_0 )=p$ for some
$t_0 \zpe (a,b)$. Then there exists $\ze>0$ such that 
$\zg(t_0 -\ze,t_0 +\ze)\zco\Zb_n (X)$.
\end{enumerate}
\end{lemma}

\begin{proof}

Around $p$ consider a vector field $U$ as in Lemma \ref{lem-1} 
such that $U(p),Y(p)$ are linearly independent.  
Then there are coordinates 
$(x_1 ,x_2 )$ about $p\zeq 0$, whose domain $D$ can be identified to a product of 
two open intervals  $J_1 \zpor J_2$, such that $Y=\zpar/\zpar x_1$ 
and $U=\zpar/\zpar x_2 +x_1 V$.

Let $X=g_1 \zpar/\zpar x_1 +g_2 \zpar/\zpar x_2$. Then 
$${\frac{\zpar g_k } {\zpar x_1}}=fg_k ,\, k=1,2.$$

Since $f$ is continuous the general solution to the equation above is: 
$$g_k (x)=h_k (x_2 )e^\zf,\, k=1,2,$$
where $\zpar\zf/\zpar x_1 =f$ and $\zf(\{0\}\zpor J_2 )=0$. 

From the Taylor expansion at $p$ of $X$ and $U$ it follows that
$$[U,[U,\dots[U,X]\dots]](0)=\zcizq{\frac {\zpar} {\zpar x_2} }, 
\zcizq{\frac {\zpar} {\zpar x_2} },\dots
\zcizq{\frac {\zpar} {\zpar x_2} },X \zcder\dots \zcder \zcder(0)$$
for the $n$-times iterated bracket.

Note that
$$\zcizq{\frac {\zpar} {\zpar x_2} }, \zcizq{\frac {\zpar} {\zpar x_2}},\dots
\zcizq{\frac {\zpar} {\zpar x_2} },X \zcder\dots \zcder \zcder(0)
={\frac {\zpar^n g_1} {\zpar x_2^n}}(0){\frac {\zpar} {\zpar x_1}}
+{\frac {\zpar^n g_2} {\zpar x_2^n}}(0){\frac {\zpar} {\zpar x_2}}.$$

Since on $\{0\}\zpor J_2$ each $g_k =h_k$ finally one has
$${\frac {\zpar^n h_1} {\zpar x_2^n}}(0){\frac {\zpar} {\zpar x_1}}
+{\frac {\zpar^n h_2} {\zpar x_2^n}}(0){\frac {\zpar} {\zpar x_2}}
=[U,[U,\dots[U,X]\dots]](0)\znoi 0,$$
which implies the existence of two diferentiable functions $\tilde h_1 (x_2)$ and 
$\tilde h_2 (x_2)$ such that $h_k =x_2^n \tilde h_k (x_2)$, $k=1,2$, and
$\tilde h_1^2 (0)+\tilde h_2^2 (0)>0$.

Therefore by shrinking $D$ if necessary, we may suppose that at least one of these
function, say $\tilde h_\zlma$, does not have any zero. Observe that $f$ will be
differentiable if $\tilde h_\zlma e^\zf$ is differentiable because 
$\tilde h_\zlma$ is differentiable
without zeros and $\zpar\zf/\zpar x_1 =f$. 

As $g_\zlma =x_2^n \zpu(\tilde h_\zlma e^\zf )$, it follows that $g_\zlma$ is
divisible by $1,x_2 ,\dots,x_2^n$ and the respective quotient functions are at least
continuous. Moreover $g_\zlma /x^r$, $r=1,\dots,n-1$, vanish if $x_2 =0$, that
is to say on $J_1 \zpor\{0\}$.

The Taylor expansion of $g_\zlma$ transversely to $J_1 \zpor\{0\}$ leads
$$g_\zlma =\zsu_{r=0}^{n-1}x_2^r \zm_r (x_1 ) +x_2^n \zm_n (x_1 ,x_2 )$$
where each $\zm_k$, $k=1,\dots,n$ is differentiable. 

Now since $g_\zlma (J_1 \zpor\{0\})=0$ one has $\zm_0 =0$.

In turn as $g_\zlma /x_2$ equals zero on $J_1 \zpor\{0\}$ it follows $\zm_1 =0$,
and so one. Hence $\zm_0 =\dots=\zm_{n-1}=0$, which implies
$g_\zlma =x_2^n \zm_n (x_1 ,x_2 )$. Therefore $\tilde h_\zlma e^\zf =\zm_n$
is differentiable, which proves \ref{lem-2a}. 

On the other hand, as $e^\zf$ is differentiable and positive, $X$ and 
$$X'\co=e^{-\zf}X=x_2^n \left(\tilde h_1 {\frac {\zpar} {\zpar x_1}}
+\tilde h_2 {\frac {\zpar} {\zpar x_2}}\right)$$
have the same order everywhere. Thus $X$ has order $n$ at every point
of $J_1 \zpor\{0\}$ and \ref{lem-2b} becomes obvious.
\end{proof}

\begin{lemma}\mylabel{lem-3}
Under the hypotheses of Proposition \ref{pro-1} consider $p\zpe\Zb_\zinf (X)$ 
with $Y(p)\znoi 0$.  Let $\zg\co(a,b)\zfl M$ be an integral curve of $Y$ passing 
through $p$ for some $t_0 \zpe (a,b)$. Then there exists $\ze>0$ such that 
$\zg(t_0 -\ze,t_0 +\ze)\zco\Zb_\zinf (X)$.
\end{lemma}

\begin{proof}
Around $p\zeq 0$ consider coordinates $(y_1 ,y_2 )$, whose domain $E$ can be 
identified to a product of two open intervals $K_1 \zpor K_2$, such that 
$Y=\zpar/\zpar y_1$ and $X=a_1 (y_2 )e^\zr \zpar/\zpar y_1
+a_2 (y_2 )e^\zr \zpar/\zpar y_2$ where $\zpar\zr/\zpar y_1 =f$ and
$\zr(\{0\}\zpor K_2 )=0$. These coordinates exist by the same reason as in the
proof of Lemma \ref{lem-2}.

Assume the existence of a $q\zpe K_1 \zpor\{0\}$ of finite order $n$.

Since $p\zpe\Zb_\zinf$ and $e^\zr$ equals $1$ on $\{0\}\zpor K_2$, it follows that
$j_0^\zinf a_1 =j_0^\zinf a_2 =0$. Therefore $a_k (y_2 )=y_2^{n+1}b_k (y_2 )$,
$k=1,2$, where each $b_k$ is differentiable. Hence there exists a continuous vector
field $X_n$ such that $X=y_2^{n+1}X_n$; that is to say $X$ is continuously divisible
by $y_2^{n+1}$.

In turn one can find coordinates $(x_1 ,x_2 )$ around $q\zeq 0$ 
whose domain $D$ can be identify to
$J_1 \zpor J_2$ as in the proof of Lemma \ref{lem-2}, which implies that
$$X=x_2^n e^\zf \left(\tilde h_1 (x_2 ){\frac {\zpar} {\zpar x_1}}
+\tilde h_2 (x_2 ){\frac {\zpar} {\zpar x_2}}\right)$$
where $\tilde h_1 \zpar/\zpar x_1 +\tilde h_2 \zpar/\zpar x_2$ has no zero on $D$.

By shrinking $D$ if necessary, we may suppose $D\zco E$. 
Then, regarded both sets in $M$, $J_1 \zpor\{0\}$ is a subset of $K_1 \zpor\{0\}$ 
since they are traces of integral curves of $Y$ with $q$ as common point. 

On the other hand as $y_2$ vanishes on $K_1 \zpor\{0\}$ but its derivative never does,
on $D$ one has $y_2 =x_2 c(x_1 ,x_2 )$ where $c$ has no zero. This fact implies 
that $X$ on $D$ is continuously divisible by $x_2^{n+1}$ because it was 
continuously divisible by $y_2^{n+1}$.

But clearly from the expression of $X$ in coordinates $(x_1 ,x_2 )$ it follows the
non-divisibility by $x_2^{n+1}$, contradiction. In short the order of $X$ at each
point of $K_1 \zpor\{0\}$ is infinite.
\end{proof}

\begin{remark}\mylabel{rem-1}
{\rm Under the hypotheses of Proposition \ref{pro-1} the tracking function $f$ can be not 
differentiable around a flat point. For instance, on $\RR^2$ set 
$Y=x_1^4 \zpar/x_1 +\zpar/\zpar x_2$ and $X=g(x_1 )\zpar/x_1$, where 
$g(x_1 )=e^{-1/x_1}$ if $x_1 >0$, $g(x_1 )=e^{-1/x_1^2}$ 
if $x_1 <0$ and $g(0)=0$. Then
$f(x)=x_1^2 -4x_1^3$ if $x_1 >0$, $f(x)=2x_1 -4x_1^3$ if $x_1 <0$ and
$f(\{0\}\zpor\RR)=0$, which is not differentiable on $\{0\}\zpor\RR$.}
\end{remark}

\begin{proof}[Proof of Proposition \ref{pro-1}]
Let us proves the first assertion. Consider a non-constant integral curve of $Y$
(the constant case is clear)  $\zg\colon (a,b)\zfl M$.
By Lemmas \ref{lem-2} and \ref{lem-3},  $\zg^{-1}(\Zb(X))$
is open in $(a,b)$. As this set is closed too one has $\zg^{-1}(\Zb(X))=\zva$ or
$\zg^{-1}(\Zb(X))=(a,b)$. The first case is obvious; in the second one
$(a,b)=\zung_{n\zpe\NN'}\zg^{-1}(\Zb_n (X))$ where each term of this union is open.
Therefore a single term of this disjoint union is non-empty since $(a,b)$ is connected.

For the second assertion apply \ref{lem-2a} of Lemma \ref{lem-2} taking into account that 
$f$ is always differentiable on $M\menos\Zb(X)$ because, on this set, the quotient  
$[Y,X]/X$ has a meaning.
\end{proof}

\begin{proposition}\mylabel{pro-2}
On a surface $M$ consider a vector field $X$ such that $\Zb(X)\znoi\zva$ but 
$\Zb_\zinf (X)=\zva$. Then at least one of the following assertions holds:
\begin{enumerate}[label={\rm (\arabic{*})}]
\item\mylabel{pro-21} $\Zb(X)$ is a regular (embedded) $1$-submanifold.
\item\mylabel{pro-22} There exists a primary singularity of $X$. 
\end{enumerate}
\end{proposition}

\begin{proof}
Assume the non-existence of primary singularities.

Consider any $p\zpe\Zb(X)$ and a vector field $Y$ defined around $p$ with $Y(p)\znoi 0$
that tracks X. Let $U$ be a second vector field about $p$ as in Lemma \ref{lem-1} such
that $U(p),Y(p)$ are linearly independent. Then there exist coordinates 
$(x_1 ,x_2 )$, about $p\zeq 0$, whose domain $D$ can be identified to a product of two 
open intervals  $J_1 \zpor J_2$ such that $Y=\zpar/\zpar x_1$ and 
$U=\zpar/\zpar x_2 +x_1 V$.

The same reasoning as in the proof of Lemma \ref{lem-2} allows to suppose that
$$X=x_2^n e^\zf \left(\tilde h_1 {\frac {\zpar} {\zpar x_1}}
+\tilde h_2 {\frac {\zpar} {\zpar x_2}}\right)$$
with $\tilde h_1^2+\tilde h_2^2 >0$ everywhere.

Therefore $\Zb(X)\zin D$ is given by the equation $x_2 =0$, which implies that $\Zb(X)$
 is a regular $1$-submanifold.
\end{proof}

\begin{theorem}\mylabel{the-2}
Consider a vector field $X$ on a surface $M$. Assume that:
\begin{enumerate}[label={\rm (\alph{*})}]
\item\mylabel{the-2a} $\Zb_\zinf(X)=\zva$. 
\item\mylabel{the-2b} There is a connected component of $\Zb(X)$ that is not
included in a single $\Zb_n (X)$.
\end{enumerate}

Then there exists a primary singularity of $X$.
\end{theorem}

\begin{proof}
Assume there is no primary singularity.
By Proposition \ref{pro-2}, $\Zb(X)$ is a regular $1$-submanifold of $M$. 
By hypothesis there are a connected component $C$ of $\Zb(X)$ and two 
different natural numbers $m$ and $n$ such that $C$ meets $\Zb_m (X)$ and $\Zb_n (X)$.

As $C$ is a regular $1$-submanifold, Proposition \ref{pro-1} and  Lemma \ref{lem-2}
imply that each $C\zin\Zb_r (X)$, $r\zpe\NN$, is open in $C$. Therefore $C$ is a disjoint
union of a family of non-empty open sets with two or more elements hence 
not connected, contradiction.
\end{proof}

\subsection{Proof of of Theorem \ref{the-1}} \mylabel{secPr}
It consists of three steps.

{\bf 1.} Assume that there is no primary singularity in $K$. From Proposition
\ref{pro-2} applied to an isolating open set it follows that $K$ is a compact
$1$-submanifold. Notice that at least  one of its connected component is
an essential block. Therefore one may suppose that $K$ is diffeomorphic to
$S^1$ and, by shrinking $M$, that $\Zb(X)=K$.

Consider a Riemannian metric $g$ on $M$. Given $p\zpe K$ by reasoning as
before one can find coordinates $(x_1 ,x_2 )$ such that $p\zeq 0$ and
$$X=x_2^n e^\zf \left(\tilde h_1 {\frac {\zpar} {\zpar x_1}}
+\tilde h_2 {\frac {\zpar} {\zpar x_2}}\right)$$
where $\tilde h_1 \zpar/\zpar x_1 +\tilde h_2 \zpar/\zpar x_2$ has no zero.
Therefore around $p$ there exists an $1$-dimensional vector subbundle
$\E$ of the tangent bundle that is orthogonal to $X$. Such a vector subbundle is unique
because clearly it exists and  is unique outside $K$. Thus, gluing together the local 
constructions gives rise to an $1$-dimensional  vector subbundle $\E$ of $TM$ that is  
orthogonal to $X$.

{\bf 2.} If $\E$ is trivial there exists a nowhere singular vector field $V$ such that
$g(V,X)=0$. Let $\zf\co M\zfl\RR$ be a function with a sufficiently narrow compact 
support such that $\zf(K)=1$. Set $X_\zd :=X+\zd\zf V$, $\zd>0$.
Then $X_\zd$ approaches $X$ as much as desired and $\Zb(X_\zd )=\zva$, so $K$ 
is an inessential block.

{\bf 3.} Now assume that $\E$ is not trivial. There always exists a $2$-folding
covering space $\zp\co M'\zfl M$ such that the pull-back $\E'\zco TM'$ of the vector
subbundle $\E$ is trivial.

Consider the vector field $X'$ on $M'$ defined by $\zp_* (X')=X$. Then
$\Zb(X')=\zp^{-1}(K)$ and $X'$ is nowhere flat. Moreover $\E'$ is orthogonal to
$X'$ with respect to the pull-back of $g$. Now the same reasoning as in the
foregoing step shows that $\ib_{\Zb(X')}(X')=0$. But clearly
$\ib_{\Zb(X')}(X')=2\ib_K (X)$ and hence $K$ is inessential.

\section{Examples} \mylabel{secEx}

\begin{example} \mylabel{exaPL}
{\rm In this example one shows two facts. First, primary singularities can exist
even if the index of $X$ is not definable. Second, 
being nowhere flat is a weaker hypothesis than being analytic.

Consider a proper closed subset $C$ of $\RR$ and a function $\zf\co\RR\zfl\RR$ 
such that $\zf^{-1}(0)=C$. Set $X:=x^2_1 \zpar/\zpar x_1 
+x_1 \zf(x_2 )\zpar/\zpar x_2$.
Then $\Zb(X)=\{0\}\zpor\RR$, $\Zb_1 (X)=\{0\}\zpor(\RR\setminus C)$,
$\Zb_2 (X)=\{0\}\zpor C$ and $\Zb_n (X)=\zva$ for $n\znoi 1,2$, so $X$ is
nowhere flat. By Theorem \ref{the-2} the vector field 
$X$ has primary singularities.

More exactly the set $S_a$ of primary singularities of $X$ equals
$\{0\}\zpor(C\setminus {\buildrel {\circ}\over{C}})$. Indeed:
\begin{enumerate}[label={\rm (\arabic{*})}]
\item\mylabel{exaPL-1} $\zf(x_2 )\zpar/\zpar x_2$ tracks $X$ and does not vanish
on $\{0\}\zpor(\RR\setminus C)$.
\item\mylabel{exaPL-2} $\zpar/\zpar x_2$ tracks $X$ on
$\RR\zpor{\buildrel {\circ}\over{C}}$.
\end{enumerate}
 
Therefore $S_a \zco\{0\}\zpor(C\setminus{\buildrel {\circ}\over{C}})$.

Take $p=(0,c)\zpe\{0\}\zpor(C\setminus{\buildrel {\circ}\over{C}})$. Assume the existence 
around this point of a vector field $Y$ with $Y(p)\znoi 0$ that tracks $X$. Them from
Proposition \ref{pro-1} and Lemma \ref{lem-2} it follows the existence of $\ze>0$ such that   
the order of $X$ at every point of $\{0\}\zpor(c-\ze,c+\ze)$ is constant and hence $c$
belongs to the interior of $\RR\setminus C$ or to that of $C$.
Therefore $c\znope C\setminus{\buildrel {\circ}\over{C}}$ contradiction.

In short, each element of $\{0\}\zpor(C\setminus{\buildrel {\circ}\over{C}})$ is a primary 
singularity and $S_a =\{0\}\zpor(C\setminus{\buildrel {\circ}\over{C}})$.

Finally observe that if $C$ is a Cantor set, then $X$ is not analytic for any analytic
structure on $\RR^2$ since $\Zb_2 (X)=\{0\}\zpor C$ is never an analytic set.} 
\end{example}

\begin{example} \mylabel{exaSP}
{\rm In this example one gives a vector field on $S^2$, which is analytic so
with no flat points, whose zero set is a circle just with two primary singularities.

The sphere $S^2$ can be regarded as  the leaves space of the $1$-dimensional foliation
on $\RR^3\setminus\{0\}$ associated to the vector field $V=\zsu_{k=1}^3 x_k \zpar/\zpar x_k$, 
while the canonical projection $\zp\co\RR^3 \setminus\{0\}\zfl S^2$ is given 
by $\zp(x)=x/\zdbv x\zdbv$.

Every linear vector field $U'$ commutes with $V$ and can be projected by $\zp$ on a vector
field $U$ on $S^2$. Moreover $U(a)=0$, where $a=(a_1 ,a_2 ,a_3 )\zpe S^2$, if and only if 
$a$ is an eigenvector  of $U'$ regarded as an endomorphism of $\RR^3$, that is to say if
and only if 
$$\left[\zsu_{k=1}^3 a_k {\frac {\zpar} {\zpar x_k }},U'\right]
=\zl\zsu_{k=1}^3 a_k {\frac {\zpar} {\zpar x_k }}$$
for some scalar $\zl$.

Set $X\co=\zp_* (x_1 \zpar/\zpar x_2 )$. Then
$\Zb (X)=\{x\zpe S^2 \co x_1 =0\}$ is an essential block of index two since $\chi (S^2 )=2$.
By Corollary \ref{cor-1} the set $S_a$ of primary singularities of $X$ is not empty.

For determining it consider the vector field $Y\co=\zp_* (x_3 \zpar/\zpar x_2 )$. Then 
$[X,Y]=0$ because $[x_1 \zpar/\zpar x_2 ,x_3 \zpar/\zpar x_2]=0$. Moreover
$\Zb (Y)=\{x\zpe S^2 \co x_3 =0\}$.

As $Y$ tracks $X$, the vector field $Y$ is tangent to $\Zb(X)$. On the other hand
$\Zb (X)\zin\Zb (Y)=\{(0,1,0),(0,-1,0)\}$, so $S_a\zco\{(0,1,0),(0,-1,0)\}$. Since 
$F_* X=X$, where $F$ is the antipodal map, one has $F(S_a )=S_a$ and hence
$S_a=\{(0,1,0),(0,-1,0)\}$.}
\end{example}

\begin{example} \mylabel{exaFL}
{\rm Let $M$ be a connected compact surface of non-vanishing Euler characteristic.
As it is well known, on $M$ there always exist two vector fields $X,Y$ with no common
zero such that $[Y,X]=X$ (Lima \cite{Lima64}, Plante \cite{Plante86}; see
\cite{Belliart97,Turiel16} as well). 
Therefore there is no primary singularity of $X$,
{\em but there always exists a periodic regular trajectory of $Y$ 
 included in $\Zb_\zinf (X)$.}

Indeed, by Corollary \ref{cor-1} and Proposition \ref{pro-1} the set 
$\Zb_\zinf (X)$ is non-empty and
$Y$-invariant. Since $\Zb_\zinf (X)$ is compact, there always exists a minimal set 
$S\zco\Zb_\zinf (X)$ of (the action of) $Y$.

As $\Zb(X)\zin\Zb(Y)=\zva$, a generalization of the Poincar\'e-Bendixson theorem
\cite{SC} implies that $S$ is homeomorphic to a circle. In other words, there exists
a non-trivial periodic trajectory of $Y$ consisting of flat points of $X$.

More generally, given a vector field $\widehat X$ on $M$ let $\A$ be the real vector 
space of those vector fields on $M$ that track $\widehat X$. Assume that
$\Zb(\widehat X)\znoi M$ and $\Zb(\widehat X)\zin(\zing_{V\zpe\A}\Zb(V))=\zva$. Then
by Corollary \ref{cor-1} the compact set $\Zb_\zinf (\widehat X)$ is not empty and 
contains a minimal set $\widehat S$ of $\A$ (more exactly of the group of 
diffeomorphisms of $M$ spanned  by the flows of the elements of $\A$). 

Clearly $\widehat S$ is not a point. A second generalization of the 
Poincar\'e-Bendixson theorem
\cite{HO} shows that $\widehat S$ is homeomorphic to a circle.

Even more, in our case $\widehat S$ is a regular $1$-submanifold and hence diffeomorphic to
a circle. Let us see it. Take $p\zpe\widehat S$; then there is $V\zpe\A$ with $V(p)\znoi 0$.
Consider coordinates $(x_1 ,x_2 )$ around $p\zeq 0$ whose 
domain $D$ is identified in the natural 
way to a product $(-\ze,\ze)\zpor(-\ze,\ze)$ such that $V=\zpar/\zpar x_1$.

Let $\zg\co(-\zd,\zd)\zfl M$ be an integral curve of $V$ with initial condition $\zg(0)=p$. 
Then $\zg(-\zd,\zd)\zco\widehat S$. Moreover, if 
$\zd$ is sufficiently small $\zg(-\zd,\zd)$ is a relatively open subset of $\widehat S$.
Indeed, $\zg\co(-\zd,\zd)\zfl\widehat S$ will be injective so open because $\widehat S$
is a $1$-dimensional topological manifold (actually $S^1$).   
Now by shrinking $D$ and $(-\zd,\zd)$ if necessary,
we may suppose that $\zg(-\zd,\zd)\zco D$, $\zd=\ze$  and $\zg(t)=(t,0)$.
Thus $(-\ze,\ze)\zpor\{0\}=\zg(-\zd,\zd)$ is relatively open in $\widehat S$ and there exists 
an open set $E$ of $M$ such that $E\zin\widehat S=(-\ze,\ze)\zpor\{0\}$. Hence 
$\widehat S \zin(D\zin E)$ is defined by the equation $x_2 =0$ in the system of coordinates
$(D\zin E,(x_1 ,x_2 ) )$.}
\end{example}

\subsection{An example from the blowup process} \mylabel{exaPO}
In this subsection one constructs a homogeneous polynomial vector
field on $\RR^2$ whose trajectories but a finite number, let us call them 
{\em exceptional}, have the origin both as $\za$ and $\zb$-limit. Then by
blowing up the origin one obtains a new vector field on a Moebius band 
whose number of primary singularities equals half 
that of exceptional trajectories of the first vector field. 

Thus a global property on the trajectories of a vector field
becomes a semi-local property on the primary singularities of another vector field.

First some technical facts.
Denote by $\widetilde \RR^2$ the surface obtained by blowing up the origin of $\RR^2$
and by $\tilde p\co\widetilde \RR^2\zfl\RR^2$ the canonical projection. 
Recall that $\widetilde \RR^2$ is a Moebius band. 
If $X$ is a vector field on
$\RR^2$ that vanishes at the origin, the blowup process gives rise to a vector field 
$\widetilde X$ on $\widetilde\RR^2$ such that $\tilde p_* \widetilde X=X$. When the 
origin is an isolated singularity of index $k$ and the order of $X$ at this point is $\zmai 2$,
then $\tilde p^{-1}(0)$ is a $\widetilde X$-block of index $k-1$.

Now identify $\CC$ to $\RR^2$ by setting $z=x_1 +ix_2$. Then each complex vector field
$z^n \zpar/\zpar z$, $n\zmai 2$, can be considered as a vector field 
$X_n =P_n \zpar/\zpar x_1 +Q_n \zpar/\zpar x_2$ on $\RR^2$ 
where $z^n =(x_1 +ix_2 )^n 
=P_n (x_1 ,x_2 )+iQ_n (x_1 ,x_2 )$. Our purpose will be to show that
$\Zb(\widetilde X_n )=\tilde p^{-1}(0)$ contains $n-1$ primary singularities of 
$\widetilde X_n$. (Recall that the origin is a singularity of $X_n$ of index $n$ and hence
$\tilde p^{-1}(0)$ is a $\widetilde X_n$-block of index $n-1$.)

\subsubsection{$\widetilde\RR^2$ from another point of view} \mylabel{exaPO1}
Consider the map $\zf\co\RR\zpor S^1\zfl\RR^2$ given by 
$\zf(r,\zh)=(rcos\zh,rsin\zh)$. Then $\zf\co\RR_+ \zpor S^1\zfl\RR^2 \setminus\{0\}$ 
and $\zf\co\RR_- \zpor S^1\zfl\RR^2 \setminus\{0\}$ are diffeomorphisms, and 
$\zf(r,\zh)=\zf(r',\zh')$ with $(r,\zh),(r',\zh')\zpe(\RR\setminus\{0\})\zpor S^1$ 
if and only if $(r,\zh)=(r',\zh')$ or $(r',\zh')=(-r,\zh+\zp)$. 

Let $\sim$ be the equivalence relation on $\RR\zpor S^1$ defined by $(r,\zh)\sim(r',\zh')$  
if and only if $(r,\zh)=(r',\zh')$ or $(r',\zh')=(-r,\zh+\zp)$. Then the quotient space
$M_s\co=(\RR\zpor S^1 )/\sim$ is  a Moebius strip and the canonical projection 
$p\co\RR\zpor S^1 \zfl M_s$ is a (differentiable) covering space with two folds.
Moreover the map $\bar\zf\co M_s \zfl\RR^2$ given by
$\bar\zf(p(r,\zh))=\zf(r,\zh)$ is well defined and differentiable. 

Recall that $\tilde p^{-1}(0)=\RR P^1$ is the space of vector lines in $\RR^2$ and 
$\tilde p\co\widetilde\RR^2\setminus\tilde p^{-1}(0)
\zfl\RR^2\setminus\{0\}$ a diffeomorphism.
Now one defines $\zQ\co M_s \zfl\widetilde\RR^2$ as follows:
\begin{enumerate}[label={\rm (\alph{*})}]
\item\mylabel{aa}
$\zQ(p(r,\zh))=\tilde p^{-1}(\zf(r,\zh))$ if $r\znoi 0$,
\item\mylabel{bb}
$\zQ(p(r,\zh))$ equals the vector line of $\RR^2$ spanned by $(cos\zh,sin\zh)$ if $r=0$.
\end{enumerate}

It is easily checked that $\zQ\co M_s \zfl\widetilde\RR^2$ is a diffeomorphism and
$\tilde p\zci\zQ=\bar\zf$. Therefore $\tilde p\co\widetilde \RR^2\zfl\RR^2$ and
$\bar\zf\co M_s \zfl\RR^2$ can be identified in this way. 
For sake of simplicity in what follows $\tilde p\co\widetilde \RR^2\zfl\RR^2$ will replaced by 
$\bar\zf\co M_s \zfl\RR^2$ in our computations. Thus if $X$ is a vector field on $\RR^2$
that vanishes at the origin, then $\widetilde X$ will be the single vector field on $M_s$ such
that $\bar\zf_* \widetilde X=X$.

On the other hand $X'$ will denote the pull-back by $p$ of $\widetilde X$. Clearly 
$\zf_* X'=X$. Moreover with respect to $X'$ the index of $\{0\}\zpor S^1$ and the number
of primary singularities included in it are twice those of $\bar\zf^{-1}(0)$ relative to
$\widetilde X$. 

As a consequence, in the case of $X_n$ it will suffice to show that 
$\Zb(X'_n )=\{0\}\zpor S^1$ contains $2n-2$ singularities of $X'_n$.

\subsubsection{Computation of the primary singularities of $X'_n$} \mylabel{exaPO2}
As $\zf\co(\RR\setminus\{0\})\zpor S^1 \zfl\RR^2 \setminus\{0\}$ is a covering space any 
vector field on $\RR^2 \setminus\{0\}$ can be lifted up. Denote by $\zpar'/\zpar x_k$, $k=1,2$,
the lifted vector field of $\zpar/\zpar x_k$. Then 
$${\frac {\zpar'} {\zpar x_1}}=cos\zh{\frac {\zpar} {\zpar r}}
-r^{-1}sin\zh{\frac {\zpar} {\zpar\zh}}                          
\hskip .5truecm \text{and}\hskip .5truecm
{\frac {\zpar'} {\zpar x_2}}=sin\zh{\frac {\zpar} {\zpar r}}
+r^{-1}cos\zh{\frac {\zpar} {\zpar\zh}}.$$

Since $(rcos\zh+irsin\zh)^n =r^n cos(n\zh)+ir^n sin(n\zh)$ one has
$P_n \zci\zf=r^n cos(n\zh)$ and $Q_n \zci\zf=r^n sin(n\zh)$. Observe that on
$(\RR\setminus\{0\})\zpor S^1$ the vector field $X'_n$ is the lifted one of $X_n$, so 
$X'_n =r^n cos(n\zh)\zpar'/\zpar x_1 +r^n sin(n\zh)\zpar'/\zpar x_2$. Finally, developing the
foregoing expression of $X'_n$ and extending it by continuity to $\RR\zpor S^1$ yields:
$$X'_n=r^{n-1}\zpizq rcos((n-1)\zh){\frac {\zpar} {\zpar r}}
+sin((n-1)\zh){\frac {\zpar} {\zpar\zh}}\zpder$$

The vector field $Y=rcos((n-1)\zh)\zpar/\zpar r+sin((n-1)\zh)\zpar/\zpar\zh$ 
tracks $X'_n$ with tracking  function $(n-1)cos((n-1)\zh)$. Therefore the set $S_a$ 
of primary singularities of $X'_n$ is included in $\{0\}\zpor T_n$ where 
$T_n \co=\{\zh\zpe S^1 \co sin((n-1)\zh)=0\}$. 

On the other hand, the order of $X'_n$ at the points of  
$\{0\}\zpor(S^1\setminus T_n)$ is $n-1$ and strictly greater than $n-1$ at the points
of $\{0\}\zpor T_n$. As $T_n$ is finite, more exactly it 
has $2n-2$ elements, Proposition \ref{pro-1} and Lemma     
\ref{lem-2} imply that all the points of $\{0\}\zpor T_n$ are primary singularities.
In short $S_a=\{0\}\zpor T_n$ and hence $\Zb(\widetilde X_n )=\tilde p^{-1}(0)$ 
contains $n-1$ primary singularities.

\subsubsection{The geometric meaning of the primary singularities of 
$\widetilde X_n$} \mylabel{exaPO3}
When $n\zmai 2$ the complex flow of $z^n \zpar/\zpar z$ is   
$$\zF(z,t)=z\zcizq(1-n)tz^{n-1}+1\zcder^{\frac {1} {1-n}}$$
with initial condition $\zF(z,0)=z$. 

(Fixed $z\znoi 0$ consider as domain of the variable $t$ the open set
$D_z \co=\CC\setminus R_z$ where 
$R_z \co=\{s(n-1)^{-1}z^{1-n}\co s\zpe[1,\zinf)\}$. Note that $D_z$ is star shaped
with respect to the origin. Since $D_z$ is simply connected, the initial condition
$\zF(z,0)=z$ defines a single continuous and hence holomorphic map
$\zF(z,\quad)\co D_z \zfl\CC$. Thus the apparent ambiguity introduced by the root
of order $n-1$ is eliminated.)

On the other hand considering, in the foregoing expression of $\zF$, 
real values of $t$ only and identifying $z$ with $(x_1 ,x_2 )$
yield the real flow of $X_n$. Therefore given
$(x_1 ,x_2 )\zpe\RR^2 \setminus\{0\}$ if $z^{n-1}=(x_1 +ix_2 )^{n-1}$ is not a real
number, its $X_n$-trajectory is defined for any $t\zpe\RR$ and has the origin both as 
$\za$ and $\zw$-limit.

On the contrary when $z^{n-1}=(x_1 +ix_2 )^{n-1}$ is a real number, 
the $X_n$-trajectory of $(x_1 ,x_2 )$, as set of points, equals the open 
half-line spanned by the vector $(x_1 ,x_2 )$ and
hence one of its limits is the origin and the other one the infinity.

It is easily checked that the set of $(x_1 ,x_2 )\zpe\RR^2$ such that 
$(x_1 +ix_2 )^{n-1}\zpe\RR$ consists of $n-1$ vector lines each of them
including two exceptional trajectories. These lines regarded as elements
of $\RR P^1 =\tilde p^{-1}(0)$ are the primary singularities of $\widetilde X_n$.

\vskip .3truecm

{\em AMS Subject Classification:} 20F16, 58J20, 37F75, 37O25, 54H25.

{\em Key words:} Vector field. Singularity. Zero set. Essential block. Surface.


\end{document}